\documentclass[11pt]{amsart}  
\usepackage{amsbsy,amsmath,amsthm,amssymb,color,verbatim}
\usepackage[backrefs]{amsrefs}
\usepackage{float}
\usepackage{fullpage,cancel,hyperref}
\usepackage{tikz}
\usepackage{mathrsfs}
\usepackage{enumitem}

\def\P{\mathcal{P}}
\def\a{\alpha}
\def\hroot{{\tilde{\alpha}}}
\def\hr{{\tilde{\alpha}}}
\allowdisplaybreaks

\usepackage{relsize}
\makeatletter
\newtheorem*{rep@theorem}{\rep@title}
\newcommand{\newreptheorem}[2]{%
\newenvironment{rep#1}[1]{%
 \def\rep@title{#2~\ref{##1}}%
 \begin{rep@theorem}}%
 {\end{rep@theorem}}}
\makeatother

\makeatletter
\newtheorem*{rep@conjecture}{\rep@title}
\newcommand{\newrepconjecture}[2]{%
\newenvironment{rep#1}[1]{%
 \def\rep@title{#2~\ref{##1}}%
 \begin{rep@conjecture}}%
 {\end{rep@conjecture}}}
\makeatother

\makeatletter
\newcommand{\addresseshere}{%
  \enddoc@text\let\enddoc@text\relax
}
\makeatother

\theoremstyle{definition}

\newtheorem{theorem}{Theorem}
\newreptheorem{theorem}{Theorem}
\newtheorem{proposition}{Proposition}
\newtheorem{problem}{Problem}
\newtheorem{corollary}{Corollary}
\newtheorem*{theorem*}{Theorem}

\author{Pamela E. Harris}
\address[P. E. Harris]{Department of Mathematics and Statistics, Williams College, United States}
\email{\textcolor{blue}{\href{mailto:peh2@williams.edu}{peh2@williams.edu}}}

\author{Margaret Rahmoeller}
\address[M. Rahmoeller]{Department of Mathematics, Computer Science, and Physics, Roanoke College}
\email{\textcolor{blue}{\href{mailto:rahmoeller@roanoke.edu}{rahmoeller@roanoke.edu}}}
\thanks{}

\author{Lisa Schneider}
\address[L. Schneider]{Department of Mathematics and Computer Science, Salisbury University}
\email{\textcolor{blue}{\href{mailto:lmschneider@salisbury.edu}{lmschneider@salisbury.edu}}}
\thanks{}

\date{\today}
\title{On the asymptotic behavior of the $q$-analog of\\Kostant's partition function}%   
\begin{document}
  \maketitle
\begin{abstract}
Kostant's partition function counts the number of distinct ways to express a weight of a classical Lie algebra $\mathfrak{g}$ as a sum of positive roots of $\mathfrak{g}$.  We refer to each of these expressions as decompositions of a weight. Our main result considers an infinite family of weights, irrespective of Lie type, for which we establish a closed formula for the $q$-analog of Kostant's partition function and then prove that the (normalized) distribution of the number of positive roots in the decomposition of any of these weights converges to a Gaussian distribution as the rank of the Lie algebra goes to infinity. We also extend these results to the highest root of the classical Lie algebras and we end our analysis with some directions for future research.
\end{abstract}

% \noindent
% \textbf{MSC 2010 subject classifications:} 05E10, 22E60, 05A15 \\
% \textbf{Keywords and phrases:} Kostant's partition function, Gaussian distribution.

\section{Introduction}
A classical problem in analytic number theory is to determine the behavior of certain  distributions associated to the decompositions of positive integers, as sums of positive integers. For example, define the Fibonacci numbers as $F_n=F_{n-1}+F_{n-2}$, whenever $n\geq 3$ and $F_1=1$, $F_2=2$. Then Zeckendorf's Theorem \cite{Zeckendorf} states that the positive integers can be uniquely expressed as a sum of nonconsecutive Fibonacci numbers and such an expression is called a decomposition. Lekkerkerker \cite{Lekkerkerker} later established that if $m \in [F_n,F_{n+1})$, then the number of summands needed in the decomposition of $m$ is asymptotic to $\left(\frac{1}{2}(1-\frac{1}{\sqrt{5}})\right)n$ as $n\to \infty$. From this work, it was found that the distribution of the number of summands in decompositions of positive integers actually converges to a Gaussian \cite{Miller1}. These results have been extended to numerous other sequences of integers which allow unique decompositions of the positive integers as a sum of elements in the sequence \cites{MR3479490,MR3584005,MR3561734,MR3666531,MR3696270,MR3513859}. 

In our work, we bring the tools of analytic number theory to the study of vector partitions. In particular, we study Kostant's partition function which counts the number of ways of expressing a weight (vector) of a simple Lie algebra $\mathfrak{g}$ as a linear combination of the positive roots of $\mathfrak{g}$ (a finite set of vectors). As is standard in analytic number theory, we refer to such expressions as decompositions. 

We recall that Lusztig \cite{LL} defined the $q$-analog of Kostant's partition function \cite{KMF} as the polynomial valued function
\begin{align*}
\wp_q\left(\xi\right)=c_0+c_1q+c_2q^2+\cdots+c_kq^k
\end{align*}
where $c_i$ denotes the number of ways the weight $\xi$ can be expressed as a sum of $i$ positive roots. Hence, evaluating $\wp_q(\xi)|_{q=1}$ yields the total number of decompositions of the weight $\xi$ as a sum of positive roots. However, the $q$-analog gives us more detailed information as it keeps track of the number of positive roots used in the decompositions and plays a key role in our analysis.

%%%% MAIN RESULT %%%%%%%
Our first main result considers an
infinite family of weights of a classical Lie algebra of rank~$r$, irrespective of Lie type, for which we establish a closed formula for the $q$-analog of Kostant's partition function and then prove that the 
(normalized) distribution of the number of positive roots in the decomposition of these weights converges to a Gaussian distribution as $r\to\infty$.

\begin{theorem}\label{thm:general result}
Let $\mathfrak{g}$ be a classical Lie algebra of rank $r$, with $\{\a_1,\a_2,\ldots,\a_r\}$ a set of simple roots of $\mathfrak{g}$. If
\[\lambda=\left(\sum_{i=1}^r\alpha_{i}\right)+\sum_{i\in I}c_i\alpha_i\]
where $I=\{i_1,i_2,\ldots, i_\ell\}$ is a set of nonconsecutive integers satisfying $1< i_1<i_2<\cdots<i_\ell<
r-2$, and  $c_{i_1},c_{i_2},\ldots,c_{i_\ell}\in \mathbb{Z}_{>0}$, 
then 
\[\wp_q(\xi)=q^{m+1} (1 + q)^{r-1-2\ell} (2 + 2 q + q^2)^\ell.\]
Morever, if  $Y_{r}$ is the  random variable denoting the total number of positive roots used in the decompositions of $\lambda$, normalize $Y_{r}$ to $Y_{r}' = \left(Y_r - \mu_{r}\right)/\sigma_{r}$ where $\mu_{r}$ and $\sigma_{r}^2$ are the mean and variance of $Y_{r}$, respectively, then
\[\mu_r=\frac{r+1}{2}-\frac{1}{5}
\ell+\sum_{i=1}^\ell c_i\qquad\mbox{ and }\qquad \sigma_r^2=\frac{r-1}{4}+\frac{3}{50}\ell,\]
and $Y_{r}'$ converges in distribution to the standard normal distribution as $r \rightarrow \infty$ for every classical Lie algebra $\mathfrak{g}$ of rank $r$.
\end{theorem}

In Lie type $A$, note that when $c_i=0$ for all $1\leq i\leq {r}$, then the weight $\lambda$ defined in Theorem~\ref{thm:general result} is  the highest root. This motivates extending Theorem~\ref{thm:general result} by considering the case when $\lambda$ is the highest root of a classical Lie algebra. Harris, Insko, and Omar gave closed formulas for the value of the $q$-analog of Kostant's partition function 
on the highest root of a classical Lie algebra  in \cite{HIO}. Our  main result follows.

\begin{theorem}\label{thm:type specific}
Let $\mathcal{P}_\mathfrak{g}(q)$ be the coefficients for the generating functions for the $q$-analog of Kostant's partition function for classical Lie algebra $\mathfrak{g}$ of rank $r$. Then $\mathcal{P}_\mathfrak{g}(q)$ is asymptotically Gaussian with mean and variance given by: \\

\noindent
Type $A_r$ $\left(r\geq 1\right)$:
\begin{align*}
 \mu_{r}  \ = \  \frac{r+1}{2}\;\;\; \mbox{and}& \;\;\;\sigma_{r}^2  \ = \  \frac{r-1}{4}
\end{align*} 
Type $B_r$ $\left(r\geq 2\right)$:
\begin{align*}
 \mu_{r}  \ &= \  \displaystyle\frac{\left(5-\sqrt{5}+\left(25-13\sqrt{5}\right)r\right)\left(5-\sqrt{5}\right)^r+\left(5+\sqrt{5}+\left(25+13\sqrt{5}\right)r\right)\left(5+\sqrt{5}\right)^r}{5\left[\left(5-3\sqrt{5}\right)\left(5-\sqrt{5}\right)^r+\left(5+3\sqrt{5}\right)\left(5+\sqrt{5}\right)^r\right]} \\ \mbox{and}& \nonumber\\
 \sigma_{r}^2  \ &= \  \displaystyle\frac{20^{r+1}r^2+\left[26\left(3\sqrt{5}-7\right)\left(5-\sqrt{5}\right)^{2r}-26\left(3\sqrt{5}+7\right)\left(5+\sqrt{5}\right)^{2r}+36\cdot\frac{20^{r+1}}{5}\right]r}{-5\left[\left(5-\sqrt{3}\right)\left(5-\sqrt{5}\right)^r+\left(5+3\sqrt{5}\right)\left(5+\sqrt{5}\right)^r\right]^2}\\\nonumber
 \ &\qquad + \ \displaystyle\frac{2\left(73-25\sqrt{5}\right)\left(5-\sqrt{5}\right)^{2r}+2\left(73+25\sqrt{5}\right)\left(5+\sqrt{5}\right)^{2r}-63\cdot\frac{20^{r+1}}{5}}{-5\left[\left(5-\sqrt{3}\right)\left(5-\sqrt{5}\right)^r+\left(5+3\sqrt{5}\right)\left(5+\sqrt{5}\right)^r\right]^2}
\end{align*} 
Type $C_r$ $\left(r\geq 3\right)$:
\begin{align*}
 \mu_{r}  \ &= \  \displaystyle
 \frac{\left(\left(1-\sqrt{5}\right)+\left(7-\sqrt{5}\right)r\right)\left(5-\sqrt{5}\right)^r+\left(\left(1+\sqrt{5}\right)+\left(7+\sqrt{5}\right)r\right)\left(5+\sqrt{5}\right)^r}{10\left(\left(5-\sqrt{5}\right)^r+\left(5+\sqrt{5}\right)^r\right)} \\ \mbox{and}& \nonumber\\
 \sigma_{r}^2  \ &= \  \displaystyle\frac{\frac{20^{r+1}}{4}r^2 + \left[13\left(\left(5-\sqrt{5}\right)^{2r}+\left(5+\sqrt{5}\right)^{2r}\right)+9\cdot\frac{20^{r+1}}{5}\right]r}{25\left(\left(5-\sqrt{5}\right)^r+\left(5+\sqrt{5}\right)^r\right)^2}\\\nonumber
 \ &\qquad + \ \displaystyle\frac{\left(-21+4\sqrt{5}\right)\left(5+\sqrt{5}\right)^{2r}-\left(21+4\sqrt{5}\right)\left(5-\sqrt{5}\right)^{2r}-37\cdot20^r}{25\left(\left(5-\sqrt{5}\right)^r+\left(5+\sqrt{5}\right)^r\right)^2}
\end{align*} 
Type $D_r$ $\left(r\geq 4\right)$:
\begin{align*}
 \mu_{r}  \ &= \ \ \displaystyle\frac{\left(15-\sqrt{5}+r\left(-5+7\sqrt{5}\right)\right)\left(5-\sqrt{5}\right)^r+\left(15+\sqrt{5}-r\left(5+7\sqrt{5}\right)\right)\left(5+\sqrt{5}\right)^r}{10\sqrt{5}\left(\left(5-\sqrt{5}\right)^r-\left(5+\sqrt{5}\right)^r\right)} \\ \mbox{and}\nonumber\\
 \sigma_{r}^2  \ &= \   \displaystyle\frac{\frac{20^{r+1}}{4}r^2-\left[13\left(\left(5+\sqrt{5}\right)^{2r}+\left(5-\sqrt{5}\right)^{2r}\right)+\frac{20^{r+1}}{5}\right]r}{-25\left[\left(5+\sqrt{5}\right)^r-\left(5-\sqrt{5}\right)^r\right]^2}\\\nonumber
 \ &\qquad + \ \displaystyle\frac{\left(34-3\sqrt{5}\right)\left(5+\sqrt{5}\right)^{2r}+\left(34+3\sqrt{5}\right)\left(5-\sqrt{5}\right)^{2r}-23\cdot 20^r}{-25\left[\left(5+\sqrt{5}\right)^r-\left(5-\sqrt{5}\right)^r\right]^2}.
\end{align*} 
\end{theorem}

This paper is organized as follows. Section~\ref{sec:general result} contains the proof of Theorem~\ref{thm:general result}. We give two proofs of Theorem \ref{thm:type specific}: the first in Section \ref{sec:Lie type results} uses the formulas of Harris, Insko, and Omar along with Bender's Theorem, while the second uses the classical approach via moment generating functions. This second proof is lengthy; hence we present it in Appendix \ref{app: alt proof}. We end with Section \ref{sec:future} where we present directions for further study.

\section{Background}\label{sec:background}

We begin by recalling the positive roots of each Lie type.

\begin{itemize}[leftmargin=1in]
\item[Type $A_r$:] If $r\geq 1$, the set of simple roots is given by $\Delta=\{\alpha_1,\alpha_{2}, \cdots, \alpha_{r}\}$,
and the set of positive roots is given by
\[
\Phi^+ = \Delta \cup \{\alpha_i+\alpha_{i+1}+\cdots+\alpha_j: 1 \leq i < j \leq r\}.
\]
The highest root is given by $\hroot=\a_1+\cdots+\a_r$.

\item[Type $B_r$:] If $r\ge 2$, the set of simple roots is given by $\Delta=\{\alpha_1,\ldots,\alpha_{r}\}$ and
the set of positive roots is given by
\[
\Phi^+=\Delta\cup\{\alpha_i+\cdots+\alpha_j:1\leq i<j\leq r\}\cup\{\alpha_i+\cdots+\alpha_{j-1}+2\alpha_j+\cdots+2\alpha_r:1\leq i<j\leq r\}.
\]
where $\hroot_{B_r}=\alpha_1+2\alpha_2+\cdots+2\alpha_4$ is the highest root.
  
\item[Type $C:$] If $r\geq 3$, the set of simple roots is given by $\Delta=\{\alpha_1,\ldots,\alpha_{r}\}$ and
the set of positive roots is given by
\begin{align*}
\Phi^+=\Delta&\cup\{\alpha_i+\cdots+\alpha_{j-1}+2\alpha_j+\cdots+2\alpha_{r-1}+\alpha_r:1\leq i<j\leq r-1\}\\&\cup\{\hroot_{C_r}\}\cup\{\alpha_i+\cdots+\alpha_j:1\leq i<j\leq r\}
\end{align*}
where $\hroot_{C_r}=2\alpha_1+2\alpha_2+\cdots+2\alpha_{r-1}+\alpha_r$ is the highest root.

\item[Type $D_r$:] If $r \geq 4$, the set of simple roots is given by  $\Delta=\{\alpha_1,\ldots,\alpha_{r}\}$ and
the set of positive roots is given by
\begin{align*}
\Phi^+&=\Delta\cup\{\alpha_i+\cdots+\alpha_{j-1}:1\leq i<j\leq r\}\cup\{\alpha_i+\cdots+\alpha_{r-2}+\alpha_r:1\leq i\leq r-2\}\\
&\qquad\cup\{\alpha_i+\cdots+\alpha_{j-1}+2\alpha_{j}+\cdots+2\alpha_{r-2}+\alpha_{r-1}+\alpha_r:1\leq i<j\leq r-2\},
\end{align*}
where $\hroot_{D_r}=\alpha_1+2\alpha_2+\cdots+2\alpha_{r-2}+\alpha_{r-1}+\alpha_r$ is the highest root. 
\end{itemize}

Let the random variable  $Y_r$ denote the total number of positive roots used in the decompositions of the highest root of a chosen classical Lie algebra as sums of positive roots, and let $p_{r,k}$ denote the number of decompositions of the highest root as a sum of exactly $k$ positive roots. We use the following result in our analysis.
\begin{proposition}[\cite{DDKMMV}, Propositions 4.7, 4.8]\label{prop:1} Let $F\left(x,y\right)=\sum_{r,k\geq0}p_{r,k}x^ry^k$ be the generating function of $p_{r,k}$, and let $g_r\left(y\right)=\sum_{k=0}^rp_{r,k}y^k$ be the coefficient of $x^r$ in $F\left(x,y\right)$. Then the mean of $Y_r$ is
\begin{align*}
\mu_r&=\frac{g_r'\left(1\right)}{g_{r}\left(1\right)}
\end{align*}
and the variance of $Y_r$ is 
\begin{align*}
\sigma_r^2&=\frac{\frac{d}{dy}\left(yg'_r\left(y\right)\right)|_{y=1}}{g_r\left(1\right)}-\mu_r^2.
\end{align*}
\end{proposition}

%%%% Section 3 %%%
\section{Proof of Theorem~\ref{thm:general result}}\label{sec:general result}
We begin with the following type $A$ result.

\begin{proposition}\label{simpler q-analog} 
Let $r\geq 3$.
If $I\subset\{2,\ldots, r-1\}$ is a set of nonconsecutive integers and $\beta=\sum_{i=1}^r \a_i+\sum_{i\in I}\a_i$ is a weight of the Lie algebra of type $A_r$, then 
\[\wp_q(\beta)=q^{|I|+1}(1+q)^{r-1-2|I|}(2+2q+q^2)^{|I|}.\]
\end{proposition}

\begin{proof}
If $I=\emptyset$, then the result follows from the fact that  $\wp_q(\sum_{i=1}^r\a_i)=q(1+q)^{r-1}$. 

Suppose the formula holds for any indexing set with cardinality $n$. Consider $I$ where $|I|=n+1$ and $j=\max(I)$. Let $I'=I-\{j\}.$ First consider the case where the additional $\alpha_j$ appears as a simple root. The number of such decompositions of $\lambda$ is 
\begin{align}
q\cdot q^{n+1}(1+q)^{r-1-2n}(2+2q+q^2)^n= q^{n+2}(1+q)^{r-1-2n}(2+2q+q^2)^n,
\label{p0}    
\end{align}
where the factor of $q$ accounts for the $\a_j$ appearing as a simple root, and, by the induction hypothesis, the remaining factors are associated with taking $I'$ as the indexing set.

Next, consider the case where $\alpha_j$ does not appear as a simple root. We treat the roots $\alpha'=\alpha_{j-1}+\alpha_j$ and $\alpha''=\alpha_j+\alpha_{j+1}$ as quasi-simple roots. In other words, they cannot be separated for this count of the decompositions. Thus it suffices to find the number of ways to write $\sum_{i=1}^{j-2} \alpha_i + \sum_{i\in I'} \alpha_i + (\alpha')$ and $\sum_{i=j+2}^r \alpha_i + (\alpha'')$ as sums of positive roots, and take the product of the results.

For $\sum_{i=1}^{j-2} \alpha_i + \sum_{i\in I'} \alpha_i + (\alpha')$, we use the inductive hypothesis (acting like in type $A_{j-1}$ and treating $\a'$ as $\a_{j-1}$) to get
\begin{align}
    q^{n+1}(1+q)^{j-1-1-2n}(2+2q+q^2)^n.
\label{p1}
\end{align}
For $\sum_{i=j+2}^{r} \alpha_i + (\alpha'')$, we use the base case (acting like in type $A_{r-j}$ and treating $\a''$ as $\a_{j+1}$) to get
\begin{align}
q(1+q)^{r-j-1}.
\label{p2}
\end{align}
Then, the number of ways that $\alpha_j$ does not appear as a simple root in the decompositions of $\beta$ is obtained by taking the product of Equations \eqref{p1} and \eqref{p2} which yields
\begin{align}
q^{n+2}(1+q)^{r-3-2n}(2+2q+q^2)^n.
\label{p3}
\end{align}
Thus the number of decompositions of $\beta$ must account for the cases where $\alpha_j$ appears as a simple root and where is does not appear as a simple root. This is given by taking the sum of Equations \eqref{p0} and \eqref{p3} which yields
\begin{align*}
\wp_q(\beta)&=q^{n+2}(1+q)^{r-1-2n}(2+2q+q^2)^n+q^{n+2}(1+q)^{r-3-2n}(2+2q+q^2)^n\\ 
&= q^{n+2}(1+q)^{r-3-2n}(2+2q+q^2)^n\left[(1+q)^2+1\right]\\
&= q^{n+2}(1+q)^{r-3-2n}(2+2q+q^2)^{n+1}\\
&= q^{(n+1)+1}(1+q)^{r-1-2(n+1)}(2+2q+q^2)^{n+1}
\end{align*}
as desired.
\end{proof}

We now give a more general result.

\begin{proposition}\label{prop:new q-analog formula}
Let  $r\geq 3$ and let 
$I=\{i_1,i_2,\ldots, i_\ell\}$ be a set of nonconsecutive integers satisfying $1< i_1<i_2<\cdots<i_\ell<  r$, and  $c_{i_1},c_{i_2},\ldots,c_{i_\ell}\in \mathbb{Z}_{>0}$. If
\[\xi=\sum_{i=1}^r\alpha_{i}+
\sum_{j=1}^{\ell}c_{i_j}\alpha_{i_j}\]
is a weight of the Lie algebra of type $A_r$,
then 
\[\wp_q(\xi)=q^{m+1} (1 + q)^{r-1-2\ell} (2 + 2 q + q^2)^\ell\]
where $m=\sum_{j=1}^{\ell} c_{i_j}$. 
\end{proposition}

\begin{proof}
The result follows from taking the formula in Proposition~\ref{simpler q-analog} and multiplying by $q^{\sum_{i=1}^\ell (c_i-1)}$, which accounts for the additional simple roots that we  must use to decompose $\xi$ than what we needed to decompose $\beta$. Hence
\begin{align*}
    \wp_q(\xi)&=q^{\sum_{i=1}^\ell (c_i-1)}\cdot q^{\ell+1}(1+q)^{r-1-2\ell}(2+2q+q^2)^{\ell}\\
    &=q^{m-\ell}\cdot q^{\ell+1}(1+q)^{r-1-2\ell}(2+2q+q^2)^{\ell}\\
    &=q^{m+1}(1+q)^{r-1-2\ell}(2+2q+q^2)^{\ell}.\qedhere
\end{align*}
\end{proof}

By further restricting the set $I$ we can give a general result for all Lie types. 
\begin{proposition}\label{prop:all Lie types}
Let $\mathfrak{g}$ be a classical simple Lie algebra of rank $r\geq 5$.
Let 
$I=\{i_1,i_2,\ldots, i_\ell\}$ be a set of nonconsecutive integers satisfying $1< i_1<i_2<\cdots<i_\ell<  r-2$, and  $c_{i_1},c_{i_2},\ldots,c_{i_\ell}\in \mathbb{Z}_{>0}$. If
\[\lambda=\sum_{i=1}^r\alpha_{i}+
\sum_{j=1}^{\ell}c_{i_j}\alpha_{i_j}\]
is a weight of $\mathfrak{g}$,
then 
\[\wp_q(\lambda)=q^{m+1} (1 + q)^{r-1-2\ell} (2 + 2 q + q^2)^\ell\]
where $m=\sum_{j=1}^{\ell} c_{i_j}$. 
\end{proposition}
\begin{proof}
The result follows from Proposition~\ref{prop:new q-analog formula} and the fact that under this restriction on the index set $I$, the only positive roots one can use in decompositions of $\lambda$ are of type $A_r$.
\end{proof}

Setting $q=1$ in Proposition~\ref{prop:all Lie types} establishes the following result.
\begin{corollary}\label{cor:total formula}
Let $I$ and $\lambda$ be defined as in Proposition~\ref{prop:all Lie types}. Then
$\wp(\lambda)=2^{r-1}\left(\frac{5}{4}\right)^\ell$.
\end{corollary}

\begin{proposition}\label{prop:general result}Let $\lambda$
be defined as in Proposition~\ref{prop:all Lie types}.
If $Y_{r}$ denotes the random variable for the total number of positive roots used in the decompositions of $\lambda$, then 
the mean and variance of $Y_{r}$ are given by 
\[\mu_r=\frac{r+1}{2}-\frac{1}{5}
\ell+m\qquad\mbox{ and }\qquad \sigma_r^2=\frac{r-1}{4}+ \frac{3}{50}\ell,\]
respectively. \end{proposition}
\begin{proof}
The result follows from Proposition~\ref{prop:1} and Corollary~\ref{cor:total formula}.
\end{proof}

\begin{proposition}\label{prop:gen Gaussian}

Let $\mu_r$ and $\sigma_r^2$ be defined as in Proposition~\ref{prop:general result} and $\lambda$ be defined as in Proposition~\ref{prop:all Lie types}. Then the random variable $Y_{r}' = \left(Y_r - \mu_{r}\right)/\sigma_{r}$ converges in distribution to the standard normal distribution as $r \rightarrow \infty$.

\end{proposition}
\begin{proof}
By Proposition~\ref{prop:all Lie types}, we let $g_r(y)=y^{m+1} (1 + y)^{r-1-2\ell} (2 + 2 y + y^2)^\ell$; hence
\begin{align}\nonumber
    \log[g_r(e^n)]&=\log[(e^n)^{m+1} (1 + e^n)^{r-1-2\ell} (2 + 2 e^n + e^{2n})^\ell]\\\nonumber
    &=(m+1)\log(e^n)+(r-1-2\ell)\log (1 + e^n)+\ell
    \log (2 + 2 e^n + e^{2n})\\
    &=(m+1)\log(1+n+n^2/2)+(r-1-2\ell)\log (2+n+n^2/2) \label{eq:sub here} \\&\qquad+\ell
    \log (5 + 4 n + 3 n^2)+O(n^3).\nonumber
\end{align}
Using Taylor's series expansion for $\log(x)$ we have 
\begin{align}\label{eq:term1}
    \log(1+n+n^2/2)&=n+O(n^3)\\\label{eq:term2}
    \log(2+n+n^2/2)&=\log(2)+\frac{1}{2}n+\frac{1}{8}n^2+O(n^3)\\\label{eq:term3}
    \log(5 + 4 n + 3 n^2)&=\log(5)+\frac{4}{5}n-\frac{1}{50}n^2+O(n^3).
\end{align}
Substituting Equations \eqref{eq:term1}, \eqref{eq:term2}, and \eqref{eq:term3}, into Equation \eqref{eq:sub here} yields
\begin{align}\label{eq:formula for gr}
    \log[g_r(e^n)]&= n^2 \left(\frac{r-1}{8}-\frac{27\ell}{100}\right) + n \left(\frac{r+1}{2}-\frac{1}{5}\ell+m\right) + 
 \ell \log\left(\frac{5}{4}\right) + ( r-1) \log(2)+O(n^3).
\end{align}
By Corollary~\ref{cor:total formula} we know 
\begin{align}
    \label{eq:Formula for gr(1)}
    \log[g_r(1)]=\log\left[2^{r-1}\left(\frac{5}{4}\right)^\ell\right]=\ell\log\left(\frac{5}{4}\right)+(r-1)\log(2).
\end{align}
Hence, substituting Equations \eqref{eq:formula for gr} and \eqref{eq:Formula for gr(1)}, $\mu_r=\frac{r+1}{2}-\frac{1}{5}\ell+m$, $\sigma_{r}=\sqrt{\frac{r-1}{4}+ \frac{3}{50}\ell}$, and $n=\frac{t}{\sigma_r}$ yields
\begin{align}\nonumber
    \log(M_{Y_r'}(t))&=\log[g_r(e^n)]-\log[g_r(1)]-\frac{t\mu_r}{\sigma_r}\\
    &=\left(\frac{25 + 54 \ell - 25 r} {50 - 12 \ell - 50 r}\right)t^2+O\left(\left(\frac{t}{\sqrt{\frac{r-1}{4}+\frac{3}{50}\ell}}\right)^3\right).\label{eq:take limit here}
\end{align}
Taking the limit of Equation \eqref{eq:take limit here}, as $r\to\infty$, we have that $\log(M_{Y_r'}(t))$ converges to $\frac{1}{2}t^2$. Thus $Y_r'$ converges to the standard normal distribution as $r\to\infty$.
\end{proof}

Theorem~\ref{thm:general result} follows directly from Proposition~\ref{prop:all Lie types} and Proposition~\ref{prop:gen Gaussian}.

\section{Proof of Theorem~\ref{thm:type specific}}\label{sec:Lie type results}
In type $A_r$ we know
\begin{align}
\wp_q\left(\hroot\right)&=q\left(1+q\right)^{r-1}=q\displaystyle\sum_{i=0}^{r-1}\binom{r-1}{i}q^{i}=\displaystyle\sum_{k=1}^{r}\binom{r-1}{k-1}q^k.\label{eq1}
\end{align}
It is well-known that the binomial distribution 
converges to a standard normal, see \cite{MR3731462} for four distinct proofs of this result. 

Harris, Insko, and Omar gave closed formulas for the generating functions for the $q$-analog of Kostant's partition function for Lie algebras of Type $B$, $C$, and $D$ in \cite{HIO}, which we restate below for ease of reference.

\begin{theorem*}[\emph{Generating Functions \cite{HIO}}]\label{ThmGenFunct}
  The closed formulas for the generating functions $\sum_{r\geq1}\mathscr{P}_{B_r}(q)x^r$, $\sum_{r\geq1}\mathscr{P}_{C_r}(q)x^r$, and $\sum_{r\geq4}\mathscr{P}_{D_r}(q)x^r$, are given by
  \begin{align*}
      \sum_{r\geq1}\mathscr{P}_{B_r}(q)x^r &= \frac{qx+\left(-q-q^2\right)x^2+q^2x^3}{1-\left(2+2q+q^2\right)x+\left(1+2q+q^2+q^3\right)x^2}, \\
      \sum_{r\geq1}\mathscr{P}_{C_r}(q)x^r &= \frac{qx+\left(-q-q^2\right)x^2}{1-\left(2+2q+q^2\right)x+\left(1+2q+q^2+q^3\right)x^2}, \\
      \sum_{r\geq4}\mathscr{P}_{D_r}(q)x^r &= \frac{\left(q+4q^2+6q^3+3q^4+q^5\right)x^4 - \left(q+4q^2+6q^3+5q^4+3q^5+q^6\right)x^5}{1-\left(2+2q+q^2\right)x+\left(1+2q+q^2+q^3\right)x^2}.
  \end{align*}

\end{theorem*}

We now use Bender's Theorem (Theorem 1 from \cite{Bender}) to show that each of these generating functions have asymptotically Gaussian coefficients. In order to match the notation in \cite{Bender}, we define $f_\mathfrak{g}(z,w) = \sum_{r\geq r_\mathfrak{g}}\mathscr{P}_\mathfrak{g}(w)z^r$, where $r_\mathfrak{g}$ either equals 1 or 4 depending on the  classical Lie algebra $\mathfrak{g}$. Since these generating functions are rational functions in $w$ and $z$, we need only check conditions two and four of Theorem 1 in \cite{Bender} (following the observation after Example 3.1) for each Lie type.

\begin{corollary}\label{BGauss}
Let $\sum_{r\geq1}\mathscr{P}_{B_r}(q)x^r$ be the generating function for the $q$-analog of Kostant's partition function for Lie algebras of Type $B$. Then the coefficients of this generating function are asymptotically Gaussian.
\end{corollary}
\begin{proof}
Let $f_{B_r}(z,w) = \frac{g_{B_r}(z,w)}{P(z,w)} = \sum_{r\geq 1}\mathscr{P}_{B_r}(w)z^r$. Then the roots of $P(z,1) = 1-5z+5z^2$ are $\frac{1}{10}\left(5-\sqrt{5}\right)$ and $\frac{1}{10}\left(5+\sqrt{5}\right)$. Hence, condition two of Bender's Theorem 1 is satisfied. And $g_{B_r}\left(\frac{1}{10}\left(5-\sqrt{5}\right),1\right)\neq 0$. Hence, condition four of Bender's Theorem 1 is satisfied. 
\end{proof}

Notice that the denominators for the generating functions for Types $C$ and $D$ are the same as for Type $B$. Hence, condition two of Bender's Theorem 1 will be satisfied for all three Lie algebra types. So, for Types C and D, we need only check condition four of Bender's Theorem. 

\begin{corollary}\label{CGauss}
Let $\sum_{r\geq1}\mathscr{P}_{C_r}(q)x^r$ be the generating function for the $q$-analog of Kostant's partition function for Lie algebras of Type C. Then the coefficients of this generating function are asymptotically Gaussian.
\end{corollary}
\begin{proof}
Let $f_{C_r}(z,w) = \frac{g_{C_r}(z,w)}{P(z,w)} = \sum_{r\geq 1}\mathscr{P}_{C_r}(w)z^r$. Since, $g_{C_r}\left(\frac{1}{10}\left(5-\sqrt{5}\right),1\right)\neq 0$, condition four of Bender's Theorem 1 is satisfied. 
\end{proof}

\begin{corollary}\label{DGauss}
Let $\sum_{r\geq4}\mathscr{P}_{D_r}(q)x^r$ be the generating function for the $q$-analog of Kostant's partition function for Lie algebras of Type D. Then the coefficients of this generating function are asymptotically Gaussian.
\end{corollary}
\begin{proof}
Let $f_{D_r}(z,w) = \frac{g_{D_r}(z,w)}{P(z,w)} = \sum_{r\geq 1}\mathscr{P}_{D_r}(w)z^r$. Since, $g_{D_r}\left(\frac{1}{10}\left(5-\sqrt{5}\right),1\right)\neq 0$, condition four of Bender's Theorem 1 is satisfied. 
\end{proof}

We now compute the means and variances for these generating functions. Harris, Insko, and Omar also gave closed formulas for the value of the $q$-analog of Kostant's partition function 
on the highest root of a classical Lie algebra  in \cite{HIO}, which we restate below for ease of reference. 
\begin{corollary}[\emph{Explicit formulas \cite{HIO}}]\label{corclosed}
Let $\hroot$ denote the highest root of a Lie algebra. Let $\P_{A_r}\left(q\right)$, $\P_{B_r}\left(q\right)$, $\P_{C_r}\left(q\right)$, and $\P_{D_r}\left(q\right)$ denote $\wp_q\left(\hroot\right)$, 
in the Lie algebras of type $A_r$, $B_r$, $C_r$, and $D_r$, respectively. Then explicit formulas for the value of the $q$-analog of Kostant's partition function on the highest root of the classical Lie algebras are as follow:
\begin{align}
\text{Type $A_r$}\;\left(r\geq 1\right):&\hspace{3mm} \P_{A_r}\left(q\right)\ = \  q\left(1+q\right)^{r-1},\\
\text{Type $B_r$}\; \left(r\geq 2\right):&\hspace{3mm}\P_{B_r}\left(q\right)\ = \ b_{+}\left(q\right) \cdot \left(\beta_+\left(q\right)\right)^{r-2} + b_{-}\left(q\right) \cdot \left(\beta_{-}\left(q\right)\right)^{r-2},\label{eq:B}\\
\text{Type $C_r$}\; \left(r\geq 1\right):&\hspace{3mm}\P_{C_r}\left(q\right)\ = \ c_{+}\left(q\right) \cdot \left(\beta_+\left(q\right)\right)^{r-1} + c_{-}\left(q\right) \cdot \left(\beta_{-}\left(q\right)\right)^{r-1},\label{eq:C}\\
\text{Type $D_r$}\; \left(r\geq 4\right):&\hspace{3mm}\P_{D_r}\left(q\right)\ = \ d_{+}\left(q\right) \cdot \left(\beta_{+}\left(q\right)\right)^{r-4} + d_{-}\left(q\right) \cdot \left(\beta_{-}\left(q\right)\right)^{r-4},\label{eq:D}
\end{align}
where 
\[
\beta_{\pm}\left(q\right) \ = \  \dfrac{\left(q^2+2q+2\right)\pm q\sqrt{q^2+4}}{2}
\]
and 
\begin{align*}
b_{\pm}\left(q\right) & =  \dfrac{\left(q^5+q^4+5q^3+4q^2+4q\right)\pm\left(q^4+q^3+3q^2+2q\right)\sqrt{q^2+4}}{2\left(q^2+4\right)}, \\
c_{\pm}\left(q\right) & =   \dfrac{\left(q^3+4q\right) \pm q^2\sqrt{q^2+4}}{2\left(q^2+4\right)} ,\\ 
d_{\pm}\left(q\right) & =  
\dfrac{q^7+3q^6+10q^5+16q^4+25q^3+16q^2+4q\pm\left(q^6+3q^5+8q^4+12q^3+9q^2+2q\right)\sqrt{q^2+4}}{2\left(q^2+4\right)}.
\end{align*}
\end{corollary}

In type $A_r$ we know
\begin{align}
\wp_q\left(\hroot\right)&=q\left(1+q\right)^{r-1}=q\displaystyle\sum_{i=0}^{r-1}\binom{r-1}{i}q^{i}=\displaystyle\sum_{k=1}^{r}\binom{r-1}{k-1}q^k.\label{eq1}
\end{align}

For $1\leq k\leq r$ let $p_{r,k}$ denote the number of ways to write $\hr$ as a sum of exactly $k$ positive roots. Hence $p_{r,k}=[\P_\mathfrak{g}\left(q\right)]_{q^k}$ and from Equation \eqref{eq1} we know that $p_{r,k}=\binom{r-1}{k-1}$ ways to write $\hr$ as a sum of exactly $k$ positive roots. We note that $p_{r,0}=0$ for all $r$, and $p_{0,k}=0$ for all $k$.

\begin{proposition}
Let $F\left(x,y\right)=\displaystyle\sum_{r\geq 0}\sum_{k\geq 0}p_{r,k}x^ry^k$ be the generating function for the $p_{r,k}'s$ arising from the number of ways to write the highest root $\hr$ as a sum of the positive roots of $\mathfrak{sl}_{r+1}\left(\mathbb{C}\right)$. Then 
\begin{align}
F\left(x,y\right)&=\frac{xy}{1-x-xy}.\label{A:Fxy}
\end{align}
\end{proposition}
\begin{proof}
The result follows from the bivariate generating function of the binomial coefficients
\[\displaystyle\sum_{n\geq 0}\sum_{k\geq 0}\binom{n}{k}x^ny^k=\frac{1}{1-y-yx}.\qedhere\]
\end{proof}

%% TYPE A MEAN AND VARIANCE %%%%%
\begin{proposition}\label{prop:2}For $r\geq 1$, the mean and variance of $Y_{A_r}$ are given by 
$ \mu_r=\frac{r+1}{2}$ and $\sigma_r^2=\frac{r-1}{4}$, respectively. \end{proposition}
\begin{proof}
By Equation \eqref{A:Fxy} and use of  the geometric sum formula, we note
\[F\left(x,y\right)=xy\left(\frac{1}{1-\left(1+y\right)x}\right)=\displaystyle\sum_{m=0}^{\infty}y\left(1+y\right)^mx^{m+1}.\]
Hence $g_r\left(y\right)=y\left(1+y\right)^{r-1}.$ Now observe that by Proposition \ref{prop:1}
\begin{align*}
\mu_r=\frac{g_r'\left(1\right)}{g_r\left(1\right)}=\frac{y\left(r-1\right)\left(1+y\right)^{r-2}+\left(1+y\right)^{r-1}|_{y=1}}{2^{r-1}}=\frac{r+1}{2}.
\end{align*}
The variance follows from a similar calculation.
\end{proof}

%%% TYPE B MEAN AND VARIANCE %%%%%%

\begin{proposition}\label{prop:B} 
For $r\geq 2$, the mean and variance of $Y_{B_r}$ are given by
 \begin{align*}
 \mu_{r}  \ &= \  \displaystyle\frac{\left(5-\sqrt{5}+\left(25-13\sqrt{5}\right)r\right)\left(5-\sqrt{5}\right)^r+\left(5+\sqrt{5}+\left(25+13\sqrt{5}\right)r\right)\left(5+\sqrt{5}\right)^r}{5[\left(5-3\sqrt{5}\right)\left(5-\sqrt{5}\right)^r+\left(5+3\sqrt{5}\right)\left(5+\sqrt{5}\right)^r]} \\ \mbox{and}& \nonumber\\
 \sigma_{r}^2  \ &= \  \displaystyle\frac{20^{r+1}r^2+[26\left(3\sqrt{5}-7\right)\left(5-\sqrt{5}\right)^{2r}-26\left(3\sqrt{5}+7\right)\left(5+\sqrt{5}\right)^{2r}+36\cdot\frac{20^{r+1}}{5}]r}{-5[\left(5-\sqrt{3}\right)\left(5-\sqrt{5}\right)^r+\left(5+3\sqrt{5}\right)\left(5+\sqrt{5}\right)^r]^2}\\\nonumber
 \ &\ \ + \ \displaystyle\frac{[2\left(73-25\sqrt{5}\right)\left(5-\sqrt{5}\right)^{2r}+2\left(73+25\sqrt{5}\right)\left(5+\sqrt{5}\right)^{2r}-63\cdot\frac{20^{r+1}}{5}]}{-5[\left(5-\sqrt{3}\right)\left(5-\sqrt{5}\right)^r+\left(5+3\sqrt{5}\right)\left(5+\sqrt{5}\right)^r]^2}.
\end{align*} 
\end{proposition}
\begin{proof}
Applying the result in Proposition \ref{prop:1} to Equation \eqref{eq:B} yields the desired result, albeit after some straightforward, but lengthy calculations.
\end{proof}

%%% TYPE C MEAN AND VARIANCE %%%%%%

\begin{proposition}\label{prop:C} For $r\geq 3$, the mean and variance of $Y_{C_r}$ are given by
\begin{align*}
 \mu_{r}  \ &= \  \displaystyle \frac{\left(\left(1-\sqrt{5}\right)+\left(7-\sqrt{5}\right)r\right)\left(5-\sqrt{5}\right)^r+\left(\left(1+\sqrt{5}\right)+\left(7+\sqrt{5}\right)r\right)\left(5+\sqrt{5}\right)^r}{10\left(\left(5-\sqrt{5}\right)^r+\left(5+\sqrt{5}\right)^r\right)} \\ \mbox{and}& \nonumber\\
 \sigma_{r}^2  \ &= \  \displaystyle\frac{\frac{20^{r+1}}{4}r^2 + \left[13\left(\left(5-\sqrt{5}\right)^{2r}+\left(5+\sqrt{5}\right)^{2r}\right)+9\cdot\frac{20^{r+1}}{5}\right]r}{25\left(\left(5-\sqrt{5}\right)^r+\left(5+\sqrt{5}\right)^r\right)^2}\\\nonumber
 \ &\ \ + \ \displaystyle\frac{\left(-21+4\sqrt{5}\right)\left(5+\sqrt{5}\right)^{2r}-\left(21+4\sqrt{5}\right)\left(5-\sqrt{5}\right)^{2r}-37\cdot20^r}{25\left(\left(5-\sqrt{5}\right)^r+\left(5+\sqrt{5}\right)^r\right)^2}
\end{align*} 
\end{proposition}
\begin{proof}
Applying the result in Proposition \ref{prop:1} to  Equation \eqref{eq:C} yields the desired result, albeit after some straightforward, but lengthy, calculations.
\end{proof}

%%% TYPE D MEAN AND VARIANCE %%%%%%

\begin{proposition}\label{prop:D}The mean and variance of $Y_{D_r}$ are given by 
\[ \mu_r= \displaystyle\frac{\left(15-\sqrt{5}+r\left(-5+7\sqrt{5}\right)\right)\left(5-\sqrt{5}\right)^r+\left(15+\sqrt{5}-r\left(5+7\sqrt{5}\right)\right)\left(5+\sqrt{5}\right)^r}{10\sqrt{5}\left(\left(5-\sqrt{5}\right)^r-\left(5+\sqrt{5}\right)^r\right)}\] 
and 
\begin{align*}
\sigma_r^2 &=\displaystyle\frac{\frac{20^{r+1}}{4}r^2-\left[13\left(\left(5+\sqrt{5}\right)^{2r}+\left(5-\sqrt{5}\right)^{2r}\right)+\frac{20^{r+1}}{5}\right]r}{-25\left[\left(5+\sqrt{5}\right)^r-\left(5-\sqrt{5}\right)^r\right]^2} \\
& \qquad + \frac{\left(34-3\sqrt{5}\right)\left(5+\sqrt{5}\right)^{2r}+\left(34+3\sqrt{5}\right)\left(5-\sqrt{5}\right)^{2r}-23\cdot 20^r}{-25\left[\left(5+\sqrt{5}\right)^r-\left(5-\sqrt{5}\right)^r\right]^2},
\end{align*}

respectively. \end{proposition}
\begin{proof}
Applying the result in Proposition \ref{prop:1} to Equation \eqref{eq:D} yields the desired result, albeit after some straightforward, but lengthy, calculations.
\end{proof}

\section{Future work}\label{sec:future}
As we have shown, the distribution of the number of positive roots used in the decompositions of certain weights, including the highest root, of the classical Lie algebras converges to a Gaussian. Hence we pose the following:
\begin{problem}\label{prob:1}
Give necessary and sufficient conditions on the weight $\mu$ of a Lie algebra of rank $r$, such that the (normalized) distribution of the number of positive roots used in the decompositions of $\mu$ as a sum of positive roots converges to a Gaussian distribution as $r\to\infty$.
\end{problem}
We point the reader to \cite{BKS}, which may provide some possible tools to answer Problem~\ref{prob:1}.

\section*{Acknowledgements}
The authors thank Diana Davis for preliminary conversations on the topic of this project.

\addresseshere

\appendix
\section{Alternate Proof of Theorem \ref{thm:type specific}}\label{app: alt proof}

To prove Theorem \ref{thm:type specific}, we  proceed via a case-by-case analysis and establish that the moment generating function of $Y_{\mathfrak{g},r}'$ converges to that of the standard normal, which is $e^{t^{2}/2}$.

\subsection{Type \texorpdfstring{$A$}{A}}\label{sec:A}
First we prove Theorem \ref{thm:type specific} for the Lie algebra of type $A_r$.
\begin{reptheorem}{thm:type specific}[Type $A_r$]
Let $\mu_{r}$ and $\sigma_r^2$ be defined as in Proposition \ref{prop:2}. Then the random variable $Y'_r=\left(Y_r-\mu_r\right)/\sigma_r^2$ converges to the standard Gaussian distribution as $r\to\infty$.
\end{reptheorem}

\begin{proof}
Recall $g_r\left(y\right)=y\left(1+y\right)^{r-1}$, hence 
\begin{align}
\log[g_r\left(e^n\right)]&=\log[e^n\left(1+e^n\right)^{r-1}]\nonumber\\
&=\log\left(e^n\right)+\left(r-1\right)\log\left(1+e^n\right)\nonumber\\
&=\log\left(1+n+\frac{n^2}{2}\right)+\left(r-1\right)\log\left(2+n+\frac{n^2}{2}\right)+O\left(n^3\right).\label{eq:2}
\end{align}
Using Taylor's series expansion for $\log\left(x\right)$ we have 
\begin{align}
\log\left(1+n+\frac{n^2}{2}\right)&=\log\left(1\right)+\frac{1}{1}\left(1+n+\frac{n^2}{2}-1 \right)-\frac{1}{1}\frac{\left( 1+n+\frac{n^2}{2}-1\right)^2}{2}+O\left(n^3\right)\nonumber\\
&=n+\frac{n^2}{2}-\frac{1}{2}\left( n+\frac{n^2}{2}\right)^2+O\left(n^3\right)\nonumber\\
&=n+\frac{n^2}{2}-\frac{1}{2}\left( n^2+n^3+\frac{n^4}{4}\right)+O\left(n^3\right)\nonumber\\
&=n+O\left(n^3\right)\label{eq:3}
\end{align}
and
\begin{align}
\log\left( 2+n+\frac{n^2}{2}\right)&=\log\left(2\right)+\frac{1}{2}\left(2+n+\frac{n^2}{2}-2 \right)-\frac{1}{4}\frac{\left( 2+n+\frac{n^2}{2}-2\right)^2}{2}+O\left(n^3\right)\nonumber\\
&=\log\left(2\right)+\frac{1}{2}\left( n+\frac{n^2}{2}\right)-\frac{1}{8}\left(n+\frac{n^2}{2}\right)^2+O\left(n^3\right)\nonumber\\
&=\log\left(2\right)+\frac{1}{2}n+\frac{1}{8}n^2+O\left(n^3\right).\label{eq:4}
\end{align}
Substituting Equations \eqref{eq:3} and \eqref{eq:4} into Equation \eqref{eq:2} yields
\begin{align}
\log[g_{r}\left(e^n\right)]&=[n+O\left(n^3\right)]+\left(r-1\right)[\log\left(2\right)+\frac{1}{2}n+\frac{1}{8}n^2+O\left(n^3\right)]+O\left(n^3\right)\nonumber\\
&=n+\left(r-1\right)\log\left(2\right)+\frac{1}{2}n\left(r-1\right)+\frac{1}{8}n^2\left(r-1\right)+O\left(n^3\right)\nonumber\\
&=n+\frac{1}{2}n\left(r-1\right)+\frac{1}{8}n^2\left(r-1\right)+O\left(n^3\right)+\log\left(g_r\left(1\right)\right)\label{eq:sub}.
\end{align}
Recall 
\begin{align}
\log\left(M_{Y_r'}\left(t\right)\right)&=\log[g_r\left(e^n\right)]-\log[g_r\left(1\right)]-\frac{t\mu_r}{\sigma_r}\label{eq:main}
\end{align}
where $\mu_r=\frac{r+1}{2}$, $\sigma_r=\sqrt{\frac{r-1}{4}}$, and $n=\frac{t}{\sigma_r}=\frac{t}{\sqrt{\frac{r-1}{4}}}=\frac{2t}{\sqrt{r-1}}.$ 
Substituting Equation \eqref{eq:sub} and $n=\frac{2t}{\sqrt{r-1}}$ into Equation \eqref{eq:main} yields
\begin{align}
\log[M_{Y'_r}\left(t\right)]&=\frac{2t}{\sqrt{r-1}}+\frac{t}{\sqrt{r-1}}\left(r-1\right)+\frac{1}{8}\left(\frac{2t}{\sqrt{r-1}}\right)^2\left(r-1\right)-\frac{t\left(\frac{r+1}{2}\right)}{\frac{\sqrt{r-1}}{2}}+O\left(\left(\frac{2t}{\sqrt{r-1}}\right)^3\right)\nonumber\\
&=\frac{1}{2}t^2+O\left(\left(\frac{2t}{\sqrt{r-1}}\right)^3\right).\label{eq:take limit here type A}
\end{align}
Taking the limit of Equation \eqref{eq:take limit here type A} as $r\to\infty$, we have that $\log(M_{Y_r'}(t))$ converges to $\frac{1}{2}t^2$. Thus $Y_r'$ converges to the standard normal distribution as $r\to\infty$.
\end{proof}

\subsection{Technical Results for other Lie Types}

For Types $B$, $C$, and $D$, we also show that $\log\left[M_{Y_{\mathfrak{g},r}'(t)}\right]$ converges to $\frac12t^2$ as $r \to \infty$, thus proving that $Y_{\mathfrak{g},r}'$ converges to the standard normal distribution as $r\to \infty$. In this section, we work through simplifying the equation for $\log\left[M_{Y_{\mathfrak{g},r}'(t)}\right]$, as many terms overlap for the various Lie types.

Recall from Corollary \ref{corclosed} that $\wp_q\left(\hroot\right)=\mathfrak{g}_{+}\left(q\right) \cdot \left(\beta_+(q)\right)^{r-i_\mathfrak{g}} + \mathfrak{g}_{-}\left(q\right) \cdot \left(\beta_-(q)\right)^{r-i_\mathfrak{g}}$,
where $\mathfrak{g}_\pm \in \left\{b_\pm,c_\pm,d_\pm\right\}$ and 
\[i_\mathfrak{g}= \begin{cases} 
      2, & \mathfrak{g}=B_r \\
      1, & \mathfrak{g}=C_r \\
      4, & \mathfrak{g}=D_r 
   \end{cases}.
\]
Hence, if we let the random variable $Y_{\mathfrak{g},r}$ denote the total number of positive roots used in the decompositions of the highest root of the Lie algebra of type $\mathfrak{g}$ as sums of positive roots, we can write $g_r(y)=\wp_q(\hroot)|_{q=y}$ and $\log\left[M_{Y_{\mathfrak{g},r}'(t)}\right] = \log[g_r(e^n)]-\log[g_r(1)]-\frac{t\mu_{\mathfrak{g},r}}{\sigma_{\mathfrak{g},r}}$. Let
\begin{align}
    M&=\mathfrak{g}_{+}\left(e^n\right) \cdot  \left(\beta_+(e^n)\right)^{r-i_\mathfrak{g}}, \label{Meqn} \\
    A&=\mathfrak{g}_{-}\left(e^n\right) \cdot \left(\beta_-(e^n)\right)^{r-i_\mathfrak{g}}, \label{Aeqn} \mbox{\qquad \qquad and }\\
    S&= A/M = \frac{\mathfrak{g}
    _-(e^n)}{\mathfrak{g}_+(e^n)}\left(\frac{\beta_-(e^n)}{\beta_+(e^n)}\right)^{r-i_\mathfrak{g}}. \label{Seqn}
\end{align}
Then 
\begin{align}
    \log\left[M_{Y_{\mathfrak{g},r}'(t)}\right] &= \log[M]+\log[1+S]-\log[g_r(1)]-\frac{t\mu_{\mathfrak{g},r}}{\sigma_{\mathfrak{g},r}}. \label{BigEqnMY}
\end{align}

We first evaluate
\begin{align*}
    \beta_{\pm}(e^n) &= \frac{(e^{2n}+2e^n+2)\pm e^n\sqrt{e^{2n}+4}}{2}.
\end{align*}
Then, using Taylor expansion of $y = e^x$ about $x=0$, we replace $e^n$ with $1+n+\frac12n^2+O(n^3)$ and obtain
\begin{align*}
    \beta_{\pm}(e^n) &= \frac{(5+4n+3n^2+O(n^3))\pm (1+n+\frac12n^2+O(n^3))\sqrt{5+2n+2n^2+O(n^3)}}{2}.
\end{align*}

By Taylor expanding $y = \sqrt{x}$ about $x = 5$ and then replacing $x = 5+2n+2n^2$ and simplifying, we get
\begin{align}
    \beta_+(e^n) &= \frac{5}{2} + \frac{\sqrt{5}}{2}+\left(2+\frac{3}{\sqrt{5}}\right)n+\left(\frac32+\frac{11}{5\sqrt{5}}\right)n^2+O(n^3) \mbox{\qquad and}\label{BetaPlusSimp}\\
    \beta_-(e^n) &=\frac{5}{2} - \frac{\sqrt{5}}{2}+\left(2-\frac{3}{\sqrt{5}}\right)n+\left(\frac32-\frac{11}{5\sqrt{5}}\right)n^2+O(n^3)\label{BetaMinusSimp}.
\end{align}

Using Equation (\ref{BetaPlusSimp}), we can rewrite
\begin{align*}
    \log[M]=\log\left[\mathfrak{g}_+(e^n)\right] + \left(r-i_\mathfrak{g}\right)\log\left[\frac{5}{2} + \frac{\sqrt{5}}{2}+\left(2+\frac{3}{\sqrt{5}}\right)n+\left(\frac32+\frac{11}{5\sqrt{5}}\right)n^2+O(n^3)\right].
\end{align*}

Now we Taylor expand $y=\log(x)$ about $x = \frac{5}{2}+\frac{\sqrt{5}}{2}$ and replace $x = \beta_+(e^n)$ as in Equation (\ref{BetaPlusSimp}) to get
\begin{align}
    \log[M]&=\log\left[\mathfrak{g}_+(e^n)\right] + \left(r-i_\mathfrak{g}\right)\left[\frac{(13+5\sqrt{5})n}{5(3+\sqrt{5})}+\frac{13n^2}{50}+\log\left(\frac12(5+\sqrt{5})\right)+O(n^3)\right].\label{LogMSimp}
\end{align}

To simplify Equation (\ref{Seqn}), we denote $N=\left(\frac{\beta_-(e^n)}{\beta_+(e^n)}\right)^{r-i_\mathfrak{g}}$, which we can simplify by Taylor expanding $y = 1/x$ about $x = \frac{5}{2} + \frac{\sqrt{5}}{2}$ and then replacing $x = \beta_{+}(e^n)$ as in Equation (\ref{BetaPlusSimp}) and multiplying this Taylor expansion to Equation (\ref{BetaMinusSimp}). We obtain
\begin{align*}
    N &= \left(\frac{1}{(5+\sqrt{5})^3}\right)^{r-i_\mathfrak{g}}\left[\left(20(5+\sqrt{5})\right) + n\left(-20(1+\sqrt{5})+2(5+\sqrt{5})n\right)\right]^{r-i_\mathfrak{g}}.
\end{align*}

Using the Binomial Theorem, we have 
\begin{align}
    N&= \left(\frac{1}{(5+\sqrt{5})^3}\right)^{r-i_\mathfrak{g}}\left[\left(20(5+\sqrt{5})\right)^{r-i_\mathfrak{g}} \right. \label{Neqn}\\
    &\left.+(r-i_\mathfrak{g})\left(20(5+\sqrt{5})\right)^{r-i_\mathfrak{g}-1} \left(n\left(-20(1+\sqrt{5})+2(5+\sqrt{5})n\right)\right)\right. \nonumber \\
    &\left.+\frac{(r-i_\mathfrak{g})(r-i_\mathfrak{g}-1)}{2}\left(20(5+\sqrt{5})\right)^{r-i_\mathfrak{g}-2}\left(n\left(-20(1+\sqrt{5})+2(5+\sqrt{5})n\right)\right)^2 + O(n^3)\right]. \nonumber
\end{align}

In order to completely describe Equation (\ref{BigEqnMY}), we need to work through each Lie algebra type separately.

%%%%%% TYPE C %%%%%%%
\subsubsection{Type $C$} 
Next we prove Theorem \ref{thm:type specific} for the Lie algebra of type $C_r$.
\begin{reptheorem}{thm:type specific}[Type $C_r$]
\label{them:GaussianC}
Let $\mu_{r}$ and $\sigma_r^2$ be defined as in Proposition \ref{prop:D}. Then the random variable $Y'_{C_r}=\left(Y_{C_r}-\mu_r\right)/\sigma_r^2$ converges to the standard Gaussian distribution as $r\to\infty$.
\end{reptheorem}

\begin{proof}
Recall 
\begin{align}\label{typeC:coeffs}
    c_{\pm}\left(q\right) \ &= \ 
\dfrac{\left(q^3+4q\right) \pm q^2\sqrt{q^2+4}}{2\left(q^2+4\right)}.
\end{align}
Replacing $e^n$ with $1+n+\frac12 n^2 + O(n^3)$ gives
\begin{align*}
    c_{\pm}\left(e^n\right) &= \frac{\left(5+7n+\frac{13n^2}{2}+ O(n^3)\right)\pm\left(1+2n+2n^2+ O(n^3)\right)\sqrt{5+2n+2n^2+ O(n^3)}}{10+4n+4n^2+ O(n^3)}.
\end{align*}

Now, Taylor expanding $y=\sqrt{x}$ about $x=5$ and $z = 1/x$ about $x = 10$, and replacing $x = 5+2n+2n^2$ and $x = 10+4n+4n^2$, respectively, gives
\begin{align}
    c_{+}\left(e^n\right) &= \left(\left(5+7n+\frac{13n^2}{2}\right)+\left(1+2n+2n^2\right)\left(\sqrt{5}+\frac{n}{\sqrt{5}}+\frac{9n^2}{10\sqrt{5}}\right)\right)\left(\frac{1}{10}-\frac{n}{25}-\frac{3n^2}{125}\right)+O(n^3) \nonumber \\
    &= \frac{1}{500}\left(50(5+\sqrt{5})+10(25+9\sqrt{5})n+(125+73\sqrt{5})n^2\right)+O(n^3). \label{CPlus}
\end{align}
Similarly,
\begin{align}
    c_{-}\left(e^n\right) &= \frac{1}{500}\left(50(5-\sqrt{5})+10(25-9\sqrt{5})n+(125-73\sqrt{5})n^2\right)+O(n^3). \label{CMinus}
\end{align}
Taylor expanding $y = 1/x$ about $x = \frac{1}{500}(50(5+\sqrt{5}))$, replacing $x = C_+(e^n)$ as in Equation (\ref{CPlus}), and multiplying the resulting expression to Equation (\ref{CMinus}) gives \begin{align}
    \frac{c_-(e^n)}{c_+(e^n)} &= \frac{1}{(5+\sqrt{5})^3}\left[20(5+\sqrt{5})-40(1+\sqrt{5})n-8(-3+\sqrt{5})n^2\right]+O(n^3). \label{CpmRatio}
\end{align}
Taylor expanding $y = \log(x)$ about $x = \frac{1}{500}(50(5+\sqrt{5}))$ and replacing $x = C_+(e^n)$ as in Equation (\ref{CPlus}) gives
\begin{align}
    \log\left(c_+(e^n)\right) &= \frac{1}{25(3+\sqrt{5})}\left[(85+35\sqrt{5})n+2(-1+\sqrt{5})n^2\right]\label{LogCPlus} \\
    & \mbox{\qquad} -\log(2)-\log(5)+\log(5+\sqrt{5})+O(n^3). \nonumber
\end{align}
Now, we can substitute Equation (\ref{LogCPlus}) into Equation (\ref{LogMSimp}), where $i_\mathfrak{g}=1$, to get 
\begin{align}
    \log[M] &= \frac{(4+2\sqrt{5}+13r+5\sqrt{5}r)n}{5(3+\sqrt{5})}+\frac{(-43-9\sqrt{5}+13(3+\sqrt{5})r)n^2}{50(3+\sqrt{5})}  \label{logMTypeC} \\
    & \mbox{\qquad} -r\log(2)-\log(5)+r\log(5+\sqrt{5})+O(n^3). \nonumber
\end{align}
Substituting Equations (\ref{CpmRatio}) and (\ref{Neqn}) (with $i_\mathfrak{g} = 1$) into Equation (\ref{Seqn}), we get an equation for $1+S$:
\begin{align}
    1+S &= 1 + 4^{1+r} 5^{r-1/2} (5+\sqrt{5})^{-3-2 r} \left[50 (2+\sqrt{5})-10 (5+2 \sqrt{5}) n (1+r)\right.\label{Insane1+S}\\\nonumber
    &\left.\qquad+n^2 (-30-11 \sqrt{5}+10 (2+\sqrt{5}) r+5 (2+\sqrt{5}) r^2)\right] + O(n^3).
\end{align}
Next, we can Taylor expand $y = \log(x)$ about $x =  1 + 4^{1+r} 5^{r-1/2} (5+\sqrt{5})^{-3-2 r}50 (2+\sqrt{5}) $ and replace $x = 1+S$ as in Equation (\ref{Insane1+S}) to obtain
\begin{align}
    \log(1+S)&= \label{InsaneLog1+S}
 \log\left[1 + 20^r (5 + \sqrt{5})^{-2 r}\right]-
 \frac{  2^{ 2 r-1} 5^{r-2}
    n}{(20^r + (5 + \sqrt{5})^{2 r})^2)}
 \left[10 \sqrt{5} (20^r + (5 + \sqrt{5})^{2 r}) (1 + r)  \right.\\
 &+\left.
     n \left(2^{3 + 2 r} 5^{1/2 + r} + (5 + \sqrt{5})^{2 r} (-5 + 8 \sqrt{5}) - 10 (5 + \sqrt{5})^{2 r} r - 
        5 (5 + \sqrt{5})^{2 r} r^2\right)\right]+O(n^3).\nonumber 
\end{align}
Given Equations (\ref{logMTypeC}) and (\ref{InsaneLog1+S}), $g_r(1)=\frac{1}{5}\left(\frac{5+\sqrt{5}}{2}\right)^r+\frac{1}{5}\left(\frac{5-\sqrt{5}}{2}\right)^r$, $n=t/\sigma_r$, and $\mu_r$ and $\sigma_r$ as in Proposition \ref{prop:C}, we find that
\begin{align*}
    \log{\left(M_{Y_r'(t)}\right)} &=\log[g_r(e^n)]-\log[g_r(1)]-\frac{t\mu_r}{\sigma_r}=k_0+k_1 t+k_2 t^2+O\left(\left(\frac{t}{\sigma_r}\right)^3\right)
\end{align*}
where $k_0 = 0$, $k_1=0$, and
\begin{align*}
    k_2&= \frac{\left((5-\sqrt{5})^r+(5-\sqrt{5})^r\right)^2}{2(3+\sqrt{5})\left(20^r+(5+\sqrt{5})^{2r}\right)^2} \left[-37\cdot20^r + (4\sqrt{5}-21)(5+\sqrt{5})^{2r} -(5-\sqrt{5})^{2r}(21+4\sqrt{5}) \right.\\
    &\nonumber \qquad \left.+ r\left(36\cdot20^r+13(5-\sqrt{5})^{2r}+13(5+\sqrt{5})^{2r}\right) + 5\cdot 20^r r^2 \right]^{-1} \\
    & \nonumber \qquad 
    \left[-37\cdot 20^r (3+\sqrt{5})(5+\sqrt{5})^{2r} -(5+\sqrt{5})^{4r}(43+9\sqrt{5})-400^r(83+33\sqrt{5}) \right.\\
    & \nonumber \qquad \left. + r(3+\sqrt{5})\left(13\cdot 400^r+36\cdot 20^r(5+\sqrt{5})^{2r}+13(5+\sqrt{5})^{4r} \right) \right. \\
    & \qquad \left. + r^2 5^{r+1}(3+\sqrt{5})\left(2(5+\sqrt{5})\right)^{2r} \right].
\end{align*}
Lastly, note $\displaystyle\lim_{r\to\infty}\log{(M_{Y_r'(t)})}=\frac{1}{2}t^2$. Thus $Y_r'$ converges to the standard normal as~$r\to\infty$.
\end{proof}

%%%%%% TYPE B %%%%%%%
\subsubsection{Type $B$} 

Our first result is as follows.
\begin{reptheorem}{thm:type specific}[Type $B_r$]
\label{them:GaussianB}
Let $\mu_{r}$ and $\sigma_r^2$ be defined as in Proposition \ref{prop:B}. Then the random variable $Y'_{B_r}=\left(Y_{B_r}-\mu_r\right)/\sigma_r^2$ converges to the standard Gaussian distribution as $r\to\infty$.
\end{reptheorem}

\begin{proof}
Recall 
\begin{align}\label{typeB:coeffs}
   b_{\pm}\left(q\right) \ &= \ 
   \dfrac{\left(q^5+q^4+5q^3+4q^2+4q\right) \pm \left(q^4+q^3+3q^2+2q\right)\sqrt{q^2+4}}{2\left(q^2+4\right)}.
\end{align}
As with Type C, replacing $e^n$ with $1+n+\frac12 n^2 + O(n^3)$, Taylor expanding $y=\sqrt{x}$ about $x=5$ and $z = 1/x$ about $x = 10$, and replacing $x = 5+2n+2n^2$ and $x = 10+4n+4n^2$, respectively, gives

\begin{align}
    b_{+}\left(e^n\right) &= \frac{3}{2}+\frac{7}{2\sqrt{5}}+\left(3+\frac{34}{5\sqrt{5}}\right)n+\left(\frac{7}{2}+\frac{194}{25\sqrt{5}}\right)n^2+O(n^3) \label{BPlus}
\end{align}
Similarly,
\begin{align}
    b_{-}\left(e^n\right) &= \frac{3}{2}-\frac{7}{2\sqrt{5}}+\left(3-\frac{34}{5\sqrt{5}}\right)n+\left(\frac{7}{2}-\frac{194}{25\sqrt{5}}\right)n^2+O(n^3). \label{BMinus}
\end{align}

Taylor expanding $y = 1/x$ about $x = \frac{3}{2}+\frac{7}{2\sqrt{5}}$, replacing $x = b_+(e^n)$ as in Equation (\ref{BPlus}), and multiplying the resulting expression to Equation (\ref{BMinus}) gives \begin{align}
    \frac{b_-(e^n)}{b_+(e^n)} &= \frac{1}{(15+7\sqrt{5})^3}\left[-20(15+7\sqrt{5})+60(7+3\sqrt{5})n+(38+6\sqrt{5})n^2\right]+O(n^3). \label{BpmRatio}
\end{align}
Taylor expanding $y = \log(x)$ about $x = \frac{3}{2}+\frac{7}{2\sqrt{5}}$ and replacing $x = b_+(e^n)$ as in Equation (\ref{BPlus}) gives
\begin{align}
    \log\left(b_+(e^n)\right) &= \frac{1}{5(15+7\sqrt{5})^2}\left[10(463+207\sqrt{5})n+(779+349\sqrt{5})n^2\right]\label{LogBPlus} \\
    & \mbox{\qquad} -\log(2)-\log(5)+\log(15+7\sqrt{5})+O(n^3). \nonumber
\end{align}

Now, we can substitute Equation (\ref{LogBPlus}) into Equation (\ref{LogMSimp}), where $i_\mathfrak{g}=2$, to get 
\begin{align}
    \log[M] &= \frac{\left(76+34\sqrt{5}+(568+254\sqrt{5})r\right)n}{5(123+55\sqrt{5})} + \frac{\left(-1157-517\sqrt{5}+(1599+715\sqrt{5})r\right)n^2}{50(123+55\sqrt{5})} \label{logMTypeB} \\
    & \mbox{\qquad} +(r-2)\log(5+\sqrt{5})-(r-1)\log(2)-\log(5)+\log(15+7\sqrt{5})+O(n^3) \nonumber.
\end{align}

Substituting Equations (\ref{BpmRatio}) and (\ref{Neqn}) (with $i_\mathfrak{g} = 2$) into Equation (\ref{Seqn}), we get an equation for $1+S$:
\begin{align}
    1+S &= 1 -\frac{4^{r+3}5^{r+2}(5+\sqrt{5})^{-2(2+r)}}{(15+7\sqrt{5})^3} \left[3600+1610\sqrt{5}-10(161+72\sqrt{5})n(r+1) \right.\label{Insane1+SB}\\
    &\nonumber\qquad\left.+n^2(-3182-14123\sqrt{5}+(720+322\sqrt{5})r+(360+161\sqrt{5})r^2) \right]+O(n^3).
\end{align}

Next, we can Taylor expand $y = \log(x)$ about $x =  1 -\frac{4^{r+3}5^{r+2}(5+\sqrt{5})^{-2(2+r)}}{(15+7\sqrt{5})^3} \left(3600+1610\sqrt{5}\right)$ and replace $x = 1+S$ as in Equation (\ref{Insane1+SB}) to obtain
\begin{align}\label{InsaneLog1+SB}
    \log{(1+S)} &= \log{\left[\frac{(15+7\sqrt{5})^3-4\cdot 20^{r+2}(5+\sqrt{5})^{-4-2r}(3600+1610\sqrt{5})}{(15+7\sqrt{5})^3}\right]}\\ 
    &\nonumber \qquad -\frac{1}{(5+\sqrt{5})^{8}(15+7\sqrt{5})^6\left( 2\cdot 20^{r}\left(360+161\sqrt{5})-(5+\sqrt{5})^{2r}(4935+2207\sqrt{5})\right)^3\right)} \\
    &\nonumber \qquad \cdot\left[ 2\cdot 20^{r+7}n\left( \left(10\cdot 400^r(162614600673847+72723460248141\sqrt{5})\right.\right.\right.\\
    &\nonumber \qquad \left. \left.\left. -40\cdot 20^r(5+\sqrt{5})^{2r}(557288527109761+249227005939632\sqrt{5})\right.\right. \right.\\
   &\nonumber \qquad \left. \left.\left. 10(5+\sqrt{5})^{4r}(7639424778862807+3416454622906707\sqrt{5})\right)(1+r)\right)\right.\\
   &\nonumber \qquad \left. +n^2\left(22\cdot 20^{2r}(162614600673847+72723460248141\sqrt{5}\right.\right.\\
   &\nonumber \qquad \left. \left. -2\cdot 20^r(5+\sqrt{5})^{2r}(23274560163131324+10408699734234047\sqrt{5})\right.\right.\\
    &\nonumber \qquad \left. \left. + (5+\sqrt{5})^{4r}(150985072020448219+67522576925084747\sqrt{5})\right.\right.\\
    &\nonumber \qquad \left. \left. +4\cdot 20^r(5+\sqrt{5})^{2r}(1246135029698160+557288527109761\sqrt{5})r\right.\right.\\
    &\nonumber \qquad \left. \left. -2\cdot (5+\sqrt{5})^{4r}(217082273114533535+7639424778862807\sqrt{5})r\right.\right.\\
   &\nonumber \qquad \left. \left. +2\cdot 20^r(5+\sqrt{5})^{2r}(1246135029698160+557288527109761\sqrt{5})r^2\right.\right.\\
    &\nonumber \qquad \left. \left. -2\cdot (5+\sqrt{5})^{4r}(217082273114533535+7639424778862807\sqrt{5})r^2\right)\right]+O(n^3).
\end{align}
Given Equations (\ref{logMTypeB}) and (\ref{InsaneLog1+SB}), $n=t/\sigma_r$, $\mu_r$ and $\sigma_r$ as in Proposition \ref{prop:B}, and 
{\small
\begin{align*}
    g_r(1) &= -\frac{5.12\cdot 10^{11}(16692641+7465176\sqrt{5})}{(5+\sqrt{5})^{12}(15+7\sqrt{5})^6}\left(r\log{\left(\frac{2}{5+\sqrt{5}}\right)}+4\log{(5+\sqrt{5})}-\log{(8(15+7\sqrt{5}))} \right.\\
    &\nonumber \left.\qquad-\log{\left(\frac{320000(5+\sqrt{5})^{-2(r+3)}(6460+2889\sqrt{5})\left(-20^r+(5+\sqrt{5})^{2r}\right)}{(15+7\sqrt{5})^3}\right)}\right),
\end{align*}
}
we find that
\begin{align*}
    \log{\left(M_{Y_r'(t)}\right)} &=\log[g_r(e^n)]-\log[g_r(1)]-\frac{t\mu_r}{\sigma_r}=k_0+k_1 t+k_2 t^2+O\left(\left(\frac{t}{\sigma_r}\right)^3\right)
\end{align*}
where $k_0$ and $k_1$ simplify to 0 and 
\begin{align*}
    k_2 &= -\left((779+349\sqrt{5})\left((5-\sqrt{5})^4(-5+3\sqrt{5})-(5+\sqrt{5})^r(5+3\sqrt{5})\right)^2\right)\\
    &\nonumber \qquad\cdot \left[2(15+7\sqrt{5})^2(-126\cdot 20^r+(73-25\sqrt{5})(5-\sqrt{5})^{2r}+(73+25\sqrt{5})(5+\sqrt{5})^{2r}\right.\\
    &\nonumber \qquad \left. +\left(72\cdot 20^r+13\left((5-\sqrt{5})^{2r}(-7+3\sqrt{5})-(5+\sqrt{5})^{2r}(7+3\sqrt{5})\right)\right)r+10\cdot20^rr^2\right] ^{-1}\\
    &\nonumber \qquad - \left(13\left((5-\sqrt{5})^{2r}(-5+3\sqrt{5})-(5+\sqrt{5})^r(5+3\sqrt{5})\right)^2(r-2)\right)\\
    &\nonumber \qquad \cdot\left[20(-126\cdot 20^r+(73-25\sqrt{5})(5-\sqrt{5})^{2r}+(73+25\sqrt{5})(5+\sqrt{5})^{2r}\right.\\
    &\nonumber \qquad \left. +(72\cdot 20^r+13((5-\sqrt{5})^{2r}(-7+3\sqrt{5})-(5+\sqrt{5})^{2r}(7+3\sqrt{5})))r+10\cdot20^r r^2\right] ^{-1}\\
    &\nonumber \qquad +\left[5\cdot 20^{r+7}\left((5-\sqrt{5})^r(-5+3\sqrt{5})-(5+\sqrt{5})^r(5+3\sqrt{5})\right)^2 \right.\\
    &\nonumber \qquad \left. \cdot \left(22\cdot 20^{2r}(162614600673847+72723460248141\sqrt{5})\right.\right.\\
   &\nonumber \qquad \left. -2\cdot 20^r(5+\sqrt{5})^{2r}(23274560163131324+10408699734234047\sqrt{5})\right.\\
    &\nonumber \qquad \left. +(5+\sqrt{5})^{4r} (150985072020448219+67522576925084747\sqrt{5})\right.\\
    &\nonumber \qquad \left. +4\cdot 20^r(5+\sqrt{5})^{2r}(1246135029698160+557288527109761\sqrt{5})r\right.\\
    &\nonumber \qquad \left. -2(5+\sqrt{5})^{4r}(17082273114533535+7639424778862807\sqrt{5})r\right.\\
   &\nonumber \qquad \left. +(5+\sqrt{5})^{2r}\left(2\cdot 20^r(1246135029698160+557288527109761\sqrt{5})\right)r^2\right.\\
    &\nonumber \qquad \left.\left. -(5+\sqrt{5})^{4r}(17082273114533535+7639424778862807\sqrt{5})r^2\right)\right]\\
    &\nonumber \qquad \left[(5+\sqrt{5})^8(15+7\sqrt{5})^6(2\cdot 20^r(360+161\sqrt{5})-(5+\sqrt{5})^{2r}(4935+2207\sqrt{5}))^3\right.\\
    &\nonumber \qquad \cdot \left.\left(-126\cdot 20^r+(73-25\sqrt{5})(5-\sqrt{5})^{2r}+(73+25\sqrt{5})(5+\sqrt{5})^{2r}\right.\right.\\
   &\nonumber \qquad \left.\left. + (72\cdot 20^r +13((5-\sqrt{5})^{2r}(-7+3\sqrt{5})-(5+\sqrt{5})^{2r}(7+3\sqrt{5}))r+10\cdot 20^r r^2\right) \right]^{-1}.
\end{align*}
Finally, note $\displaystyle\lim_{r\to \infty}\log{(M_{Y_r'(t)})}=\frac{1}{2}t^2$. Thus, $Y'_r$ converges to the standard normal distribution as $r\to \infty$.
\end{proof}

%%%%%% TYPE D %%%%%%%
\subsubsection{Type $D$} 
Next we prove Theorem \ref{thm:type specific} for the Lie algebra of type $D_r$.

\begin{reptheorem}{thm:type specific}[Type $D_r$]
\label{them:GaussianD}
Let $\mu_{r}$ and $\sigma_r^2$ be defined as in Proposition \ref{prop:D}. Then the random variable $Y'_{D_r}=\left(Y_{D_r}-\mu_r\right)/\sigma_r^2$ converges to the standard Gaussian distribution as $r\to\infty$.
\end{reptheorem}

\begin{proof}
Recall 
{\small\begin{align}\label{typeD:coeffs}
   d_{\pm}\left(q\right) \ &= \ 
\dfrac{\left(q^7+3q^6+10q^5+16q^4+25q^3+16q^2+4q\right)\pm\left(q^6+3q^5+8q^4+12q^3+9q^2+2q\right)\sqrt{q^2+4}}{2\left(q^2+4\right)}.
\end{align}}

As with Types B and C, replacing $e^n$ with $1+n+\frac12 n^2 + O(n^3)$, Taylor expanding $y=\sqrt{x}$ about $x=5$ and $z = 1/x$ about $x = 10$, and replacing $x = 5+2n+2n^2$ and $x = 10+4n+4n^2$, respectively, gives

\begin{align}
    d_{+}\left(e^n\right) &= \frac{15}{2}+\frac{7\sqrt{5}}{2}+\left(22+\frac{51}{\sqrt{5}}\right)n+\left(36+\frac{829}{10\sqrt{5}}\right)n^2+O(n^3) \label{DPlus}
\end{align}
Similarly,
\begin{align}
    d_{-}\left(e^n\right) &= \frac{15}{2}-\frac{7\sqrt{5}}{2}+\left(22-\frac{51}{\sqrt{5}}\right)n+\left(36-\frac{829}{10\sqrt{5}}\right)n^2+O(n^3) \label{DMinus}
\end{align}

Taylor expanding $y = 1/x$ about $x = \frac{15}{2}+\frac{7\sqrt{5}}{2}$, replacing $x = d_+(e^n)$ as in Equation (\ref{DPlus}), and multiplying the resulting expression to Equation (\ref{DMinus}) gives \begin{align}
    \frac{d_-(e^n)}{d_+(e^n)} &= \frac{1}{(15+7\sqrt{5})^3}\left[-20(15+7\sqrt{5})+20(7+3\sqrt{5})n+(54+22\sqrt{5})n^2\right]+O(n^3). \label{DpmRatio}
\end{align}
Taylor Expanding $y = \log(x)$ about $x = \frac{15}{2}+\frac{7\sqrt{5}}{2}$ and replacing $x = d_+(e^n)$ as in Equation (\ref{DPlus}) gives
\begin{align}
    \log\left(d_+(e^n)\right) &= \frac{1}{5(15+7\sqrt{5})^2}\left[10(687+307\sqrt{5})n+3(387+173\sqrt{5})n^2\right]\label{LogDPlus} \\
    & \mbox{\qquad} -\log(2)+\log(15+7\sqrt{5})+O(n^3). \nonumber
\end{align}

Now, we can substitute Equation (\ref{LogDPlus}) into Equation (\ref{LogMSimp}), where $i_\mathfrak{g}=4$, to get 
\begin{align}
    \log[M] &= \frac{20n\left(-237-106\sqrt{5}+(284+127\sqrt{5})r\right)+n^2\left(-3357-1501\sqrt{5}+13(123+55\sqrt{5})r)\right)}{50(123+55\sqrt{5})}\label{logMTypeD} \\
    & \qquad +3\log{2}-4\log{(5+\sqrt{5})}+\log{(15+7\sqrt{5})}+r\left(\log{(5+\sqrt{5})}-\log{2}\right)+O(n^3). \nonumber
\end{align}

Substituting Equations (\ref{DpmRatio}) and (\ref{Neqn}) (with $i_\mathfrak{g} = 4$) into Equation (\ref{Seqn}), we get an equation for $1+S$:
\begin{align}
    1+S &= 1 - \frac{4^{r+4}5^{r+3}(5+\sqrt{5})^{-2(3+r)}}{(15+7\sqrt{5})^3}\left[64600+28890\sqrt{5}-10(2889+1292\sqrt{5})n(r-3) \right.\label{Insane1+SD}\\
    &\nonumber\qquad\left.+n^2(40806+18249\sqrt{5}-6(6460+2889\sqrt{5})r+(6460+2889\sqrt{5})r^2) \right]+O(n^3).
\end{align}

Next, we can Taylor expand $y = \log(x)$ about $x =  1 - \frac{4^{r+4}5^{r+3}(5+\sqrt{5})^{-2(3+r)}}{(15+7\sqrt{5})^3}\left(64600+28890\sqrt{5}\right)$ and replace $x = 1+S$ as in Equation (\ref{Insane1+SD}) to obtain
{\small\begin{align}
    \log{(1+S)} &= \frac{1}{(5+\sqrt{5})^{12}(15+7\sqrt{5})^6}\left(5.12\cdot 10^{11}\left(-1292+\frac{2889}{\sqrt{5}}\right)(215668928180+96450076809\sqrt{5})\right. \label{InsaneLog1+SD} \\ 
    &\nonumber \qquad \left. \cdot \log{\left(\frac{320000(5+\sqrt{5})^{-2(r+3)}(6460+2889\sqrt{5})\left(-20^r+(5+\sqrt{5})^{2r}\right)}{(15+7\sqrt{5})^3}\right)}\right) \\
    &\nonumber \qquad -\frac{2^{2r+18}5^{r+9}\left(-1292+\frac{2889}{\sqrt{5}}\right)(96450076809+43133785636\sqrt{5})(r-3)n}{(5+\sqrt{5})^12(15+7\sqrt{5})^6\left(20^r-(5+\sqrt{5})^{2r}\right)} \\
    &\nonumber \qquad -\frac{3\cdot 2^{2r+17}5^{r+8}\left(-1292+\frac{2889}{\sqrt{5}}\right)n^2}{(5+\sqrt{5})^12(15+7\sqrt{5})^{6}\left(20^r-(5+\sqrt{5})^{2r}\right)^2} \\
    &\nonumber \qquad \cdot \left(2^{2r+1}5^r(96450076809+43133785636\sqrt{5})+(5+\sqrt{5})^{2r}(454106630922+203082659155\sqrt{5}) \right. \\ 
    &\nonumber\qquad \left. -2(5+\sqrt{5})^{2r}(215668928180+96450076809\sqrt{5})r\right)+O(n^3). 
\end{align}
}

Given Equations (\ref{logMTypeD}) and (\ref{InsaneLog1+SD}), $n=t/\sigma_r$, $\mu_r$ and $\sigma_r$ as in Proposition \ref{prop:D}, and 
{\small\begin{align*}
    g_r(1) &= -\frac{5.12\cdot 10^{11}(16692641+7465176\sqrt{5})}{(5+\sqrt{5})^{12}(15+7\sqrt{5})^6}\left(r\log{\left(\frac{2}{5+\sqrt{5}}\right)}+4\log{(5+\sqrt{5})}-\log{(8(15+7\sqrt{5}))} \right.\\
    &\nonumber \left.\qquad-\log{\left(\frac{320000(5+\sqrt{5})^{-2(r+3)}(6460+2889\sqrt{5})\left(-20^r+(5+\sqrt{5})^{2r}\right)}{(15+7\sqrt{5})^3}\right)}\right),
\end{align*}
}
we find that
\begin{align*}
    \log{\left(M_{Y_r'(t)}\right)} &=\log[g_r(e^n)]-\log[g_r(1)]-\frac{t\mu_r}{\sigma_r}=k_0+k_1 t+k_2 t^2+O\left(\left(\frac{t}{\sigma_r}\right)^3\right)
\end{align*}
where $k_0$ and $k_1$ simplify to 0 and 
\begin{align*}
    k_2 &= \frac{2.56\cdot10^{11}}{(5+\sqrt{5})^{12}(15+7\sqrt{5})^6\left(20^r-(5+\sqrt{5})^{2r}\right)^2} \\
    &\nonumber \qquad \cdot \left[23\cdot 20^r+(3\sqrt{5}-34)(5+\sqrt{5})^{2r}-(5-\sqrt{5})^{2r}(34+3\sqrt{5}) \right. \\
    &\nonumber \qquad\left.+\left(4\cdot20^{r}+13\left((5-\sqrt{5})^{2r}+(5+\sqrt{5})^{2r}\right)\right)r-5\cdot20^r r^2\right]^{-1} \\
    &\nonumber \qquad \cdot \left[389743431\cdot 2^{6r+1}5^{3r+1/2}+1742985611\cdot 8000^r \right.\\
    & \nonumber \qquad -339763717\cdot2^{4r+1}\left((25-5\sqrt{5})^{2r}+(5(5+\sqrt{5}))^{2r}\right) \\
    &\nonumber \qquad -303893907\cdot 5^{2r+1/2}16^r\left((5-\sqrt{5})^{2r} + (5+\sqrt{5})^{2r} \right)\\
    & \nonumber \qquad + 383930743(5+\sqrt{5})^{3r}\left(20^r(5+\sqrt{5})^{r} - 2^{2r+1}(25-5\sqrt{5})^r  \right) \\
    & \nonumber \qquad + 21462381\cdot 5^{r+1/2}(5+\sqrt{5})^{3r}2^{(2r+3)}\left(2(5+\sqrt{5})^{r}-2(5-\sqrt{5})^r \right)\\
    & \nonumber \qquad -455572154(5+\sqrt{5})^{4r}\left((5-\sqrt{5})^{2r}+(5+\sqrt{5})^{2r} \right)\\
    & \nonumber \qquad +911144308(5-\sqrt{5})^r(5+\sqrt{5})^{5r}+407476122\sqrt{5}(5-\sqrt{5})^r(5+\sqrt{5})^{5r} \\
    & \nonumber \qquad \left. -203738061 \sqrt{5}(5+\sqrt{5})^{4r} \left((5+\sqrt{5})^{2r} + (5-\sqrt{5})^{2r}\right) \right] \\
    & \nonumber \qquad + r \left[-10264617\cdot 4^{3r+2}5^{3r+1/2}-183619051\cdot 2^{6r+1}125^r +12130911\cdot 2^{4r+3}5^{2r+1/2}(5+\sqrt{5})^{2r}\right.\\
    & \nonumber \qquad +217004333\cdot \left(400^r\left((5-\sqrt{5})^{2r}+(5+\sqrt{5})^{2r}\right)+(5+\sqrt{5})^{6r}\right) \\
    & \nonumber \qquad 
    +16692641\cdot2^{2r+2}5^r(5+\sqrt{5})^{3r}\left((5+\sqrt{5})^{r}-2(5-\sqrt{5})^r \right) \\
    & \nonumber \qquad +933147\cdot 2^{2r+5}5^{r+1/2}(5+\sqrt{5})^{4r}
    +97047288\sqrt{5}(5+\sqrt{5})^{6r} \\
    & \nonumber \qquad -2(5-\sqrt{5})^r(5+\sqrt{5})^{3r}\left(9331472^{2r+5}5^{r+1/2}+13(5+\sqrt{5})^{2r}(16692641+7465176\sqrt{5})\right) \\
    & \nonumber \qquad \left. +13(5-\sqrt{5})^{2r}\left(933147\cdot2^{4r+3}5^{2r+1/2}+(5+\sqrt{5})^{4r}(16692641+7465176\sqrt{5})\right) \right].
\end{align*}
Finally, note $\displaystyle\lim_{r\to \infty}\log{(M_{Y_r'(t)})}=\frac12t^2$. Thus, $Y'_r$ converges to the standard normal distribution as $r\to \infty$.
\end{proof}

\end{document}